\documentclass{article}
\usepackage{fontenc}
\usepackage{amsmath}
\usepackage{amssymb}
\usepackage{amsfonts}
\usepackage{amsthm}

\newtheorem{thm}{Theorem}[section]
\newtheorem{prop}[thm]{Proposition}
\newtheorem{cor}[thm]{Corollary}

\newtheorem{lem}[thm]{Lemma}

\begin{document}

\title{On Quasi-Modular Forms, Almost Holomorphic Modular Forms, and the
Vector-Valued Modular Forms of Shimura}

\author{Shaul Zemel\thanks{The initial stage of this research has been carried
out as part of my Ph.D. thesis work at the Hebrew University of Jerusalem,
Israel. The final stage of this work was carried out at TU Darmstadt and
supported by the Minerva Fellowship (Max-Planck-Gesellschaft).}}


\maketitle

\section*{Introduction}

\addcontentsline{toc}{section}{Introduction}

Modular forms have been the subject of extensive research for a very long time.
Throughout this time, many generalizations of the classical notion were defined.
Two such notions which we consider in this paper are vector-valued modular forms
with representations, and quasi-modular forms, in which the transformation under
the action of the Fuchsian group involves more functions. Holomorphic
quasi-modular forms were defined in \cite{[KZ]} and their relation to classical
holomorphic modular forms was also considered in that reference, as well as in
\cite{[MR]}, \cite{[A]}, and other references. The first aim of this paper is to
extend this relation to more general modular forms and quasi-modular forms.

One type of vector-valued modular forms with representations arises from the
symmetric powers of the natural action of $SL_{2}(\mathbb{R})$ on
$\mathbb{C}^{2}$. The $m$th symmetric power of this representation is denoted
$V_{m}$. Modular forms with representation $V_{m}$ were defined, using
expressions similar to ours, by Shimura in \cite{[Sh]}, in order to consider
the integrals defined by Eichler in \cite{[E]} (see also the work \cite{[B]} of
Bol, where a preliminary version of these modular forms already appears). A
structure theorem for cusp forms of weights 0 and 2 and representation $V_{m}$
for even $m$ was established by Kuga and Shimura in \cite{[KS]}. The main aim
of this paper is to relate (holomorphic) quasi-modular forms to holomorphic
modular forms with the representation $V_{m}$. In fact, we obtain this result in
a more general setting. Using the tools developed in order to achieve this task,
we also prove a general structure theorem for modular forms with
representations (or multiplier systems) involving $V_{m}$.

In Section \ref{MQM} we present the notions of modular forms and quasi-modular
forms, and study the relation between them under very general assumptions. In
Section \ref{Vm} we introduce modular forms with representation $V_{m}$, obtain
the connection between all three notions, and prove the structure theorem.

I wish to thank the anonymous referee for his helpful comments on this paper.

\section{Modular Forms and Quasi-Modular Forms \label{MQM}}

In this Section we present the notions of modular and quasi-modular forms, and
give the relations between them in a general setting.

\subsection{Definitions and Notation}

Let $\mathcal{H}=\{\tau\in\mathbb{C}|\Im\tau>0\}$ be the Poincar\'{e} upper
half-plane. We shall always write $\tau=x+iy$, hence $x=\Re\tau$ and
$y=\Im\tau$. The group $SL_{2}(\mathbb{R})$ acts on
$\mathcal{H}=\{\tau\in\mathbb{C}|\Im\tau>0\}$ by M\"{o}bius transformations:
\begin{equation}
\gamma=\left(\begin{array}{cc} a & b
\\ c & d\end{array}\right):\tau\mapsto\frac{a\tau+b}{c\tau+d}. \label{action}
\end{equation}
The measure $d\mu=\frac{dxdy}{y^{2}}$ is invariant under the action of
$SL_{2}(\mathbb{R})$. A discrete subgroup $\Gamma \subseteq SL_{2}(\mathbb{R})$
is called \emph{Fuchsian of the first kind} if the volume of a fundamental
domain (with respect to $d\mu$) is finite. As this is the only subgroups of
$SL_{2}(\mathbb{R})$ we consider in this paper, we shall write simply
\emph{Fuchsian}, meaning Fuchsian of the first kind, throughout. A subgroup of
trace $\pm2$ elements of $\Gamma$ is called \emph{parabolic} and corresponds to
a \emph{cusp} of $\Gamma$, i.e., to the element of $\mathbb{P}^{1}(\mathbb{R})$
which is stabilized by this subgroup. $\Gamma$ acts on its cusps via (extended)
M\"{o}bius transformations (which correspond to conjugation on the parabolic
subgroups), and the number of orbits of cusps is always finite. The quotient
space $Y(\Gamma)=\Gamma\backslash\mathcal{H}$ is a Riemann surface, which
becomes a compact Riemann surface (or algebraic curve) denoted $X(\Gamma)$ after
adding the (finitely many) equivalence classes of cusps of $\Gamma$.

The action of $SL_{2}(\mathbb{R})$ on $\mathcal{H}$ admits a \emph{factor of
automorphy}, which is defined, for any elements $\gamma \in SL_{2}(\mathbb{R})$
and $\tau\in\mathcal{H}$ as in Equation \eqref{action}, by
\begin{equation}
j(\gamma,\tau)=c\tau+d. \label{autfac}
\end{equation}
The factor of automorphy is a cocycle: The equality
\begin{equation}
j(\gamma\delta,\tau)=j(\gamma,\delta\tau)j(\delta,\tau) \label{cocycle}
\end{equation}
holds for any $\tau\in\mathcal{H}$ and $\gamma$ and $\delta$ in
$SL_{2}(\mathbb{R})$. We shall sometimes use the alternative notation
$j_{\gamma}(\tau)$ for $j(\gamma,\tau)$. The entry $c$ of the matrix $\gamma$
from Equation \eqref{action} can be written as the derivative
$j_{\gamma}'(\tau)$ of the linear function $j_{\gamma}$ from Equation
\eqref{autfac}. Since it is a constant (independent of $\tau$), we omit the
variable and write simply $j_{\gamma}'$.

Let $k$ and $l$ be integers, let $\Gamma \subseteq SL_{2}(\mathbb{R})$ be a
Fuchsian group, and let $\rho$ be a representation of $\Gamma$ on some
(finite-dimensional) complex vector space $V_{\rho}$. Throughout this paper,
when a representation $\rho$ of some group $\Gamma$ is considered, $V_{\rho}$
denotes the representation space of $\rho$. A \emph{modular form of weight
$(k,l)$ and representation $\rho$ with respect to $\Gamma$} is a real-analytic
function $f:\mathcal{H}\to\mathbb{C}$ satisfying the functional equation
\begin{equation}
f(\gamma\tau)=j(\gamma,\tau)^{k}\overline{j(\gamma,\tau)}^{\
l}\rho(\gamma)f(\tau) \label{modrep}
\end{equation}
for any $\gamma\in\Gamma$. A modular form of weight $(k,0)$ is said to have
weight $k$. By considering subgroups of the metaplectic double cover
$Mp_{2}(\mathbb{R})$ of $SL_{2}(\mathbb{R})$ (one realization of which consists
of pairs of an element $\gamma \in SL_{2}(\mathbb{R})$ and a choice of a square
root $\sqrt{j(\gamma,\tau)}$ of the automorphy factor from Equation
\eqref{autfac}), one can consider modular forms in which the weights $k$ and $l$
are half-integral. Furthermore, we can consider the case in which the weights
$k$ and $l$ are arbitrary real (and even complex) by allowing $\rho$ to be a
\emph{multiplier system} of weight $(k,l)$. We recall that such a multiplier
system is a function $\rho:\Gamma\to\mathbb{C}$ (or, more generally,
$\rho:\Gamma \to GL(V_{\rho})$), which satisfies the condition
\[j_{\gamma\delta}^{k}(\tau)\overline{j_{\gamma\delta}^{l}(\tau)}
\rho(\gamma\delta)=j_{\gamma}^{k}(\delta\tau)\overline{j_{\gamma}^{l}
(\delta\tau)}j_{\delta}^{k}(\tau)\overline{j_{\delta}^{l}(\tau)}
\rho(\gamma)\rho(\delta)\] for any $\gamma$ and $\delta$ in $\Gamma$ and
$\tau\in\mathcal{H}$ (with the appropriate choice of powers, or equivalently
logarithms, of the automorphy factor functions $j_{\gamma}$ for
$\gamma\in\Gamma$). Note that we do not require the image of $\rho$ to be
unitary, hence the 1-dimensional case of our definition covers also the case of
``generalized modular forms'' in the sense of \cite{[KM]} (but extended to
arbitrary Fuchsian groups). Multiplication by $y^{t}$ takes a modular form of
weight $(k+t,l+t)$ (and any representation or multiplier system) to a modular
form of weight $(k,l)$ (and the same representation or multiplier system). It
follows that using this operation we can always consider only modular forms of
``holomorphic'' weights. We denote the space of (real-analytic) modular forms of
weight $(k,l)$ and representation (or multiplier system) $\rho$ by
$\mathcal{M}_{k,l}^{an}(\rho)$, and write $\mathcal{M}_{k}^{an}(\rho)$ for
$\mathcal{M}_{k,0}^{an}(\rho)$. In some cases it will be useful to allow
singularities on $\mathcal{H}$, so that the space
$\mathcal{M}_{k,l}^{sing}(\rho)$ (and $\mathcal{M}_{k}^{sing}(\rho)$) consists
of those modular forms which are real-analytic on $\mathcal{H}$ except in a
discrete ($\Gamma$-invariant) set of points. The notation
$\mathcal{M}_{k}^{hol}(\rho)$ and $\mathcal{M}_{k}^{mer}(\rho)$ stands for the
spaces consisting of modular forms which are holomorphic (resp. meromorphic) on
$\mathcal{H}$. If $\Gamma$ has cusps, then elements of
$\mathcal{M}_{k}^{hol}(\rho)$ are required to be holomorphic also at the cusps.
In this case we denote the space of cusp forms (i.e., those elements of
$\mathcal{M}_{k}^{hol}(\rho)$ which vanish at the cusps) by
$\mathcal{M}_{k}^{cusp}(\rho)$. The space of meromorphic modular forms having
all their poles at the cusps (these forms are called \emph{weakly holomorphic})
is denoted $\mathcal{M}_{k}^{wh}(\rho)$. For integral $k$ (and $l$), replacing
$\rho$ by $\Gamma$ in all these notations denotes the corresponding spaces of
modular forms with trivial (1-dimensional) representation.

Another notion, namely that of a quasi-modular form, was defined first in
\cite{[KZ]} and considered, among others, in \cite{[A]} and Section 3 of
\cite{[MR]}. These references restricted attention only to holomorphic (or
meromorphic) functions. Here we introduce a more general class of functions, by
considering real-analytic functions and allowing representations and multiplier
systems. Let $k$ be a real (or even complex) number, let $d$ be a non-negative
integer, let $\Gamma$ be a Fuchsian group, and let $\rho:\Gamma \to
GL(V_{\rho})$ be a multiplier system of weight $k$. A real-analytic function
$f:\mathcal{H} \to V_{\rho}$ is \emph{quasi-modular form of weight $k$, depth
$d$, and multiplier system $\rho$ with respect to $\Gamma$} if there exist
(real-analytic) functions $f_{r}:\mathcal{H} \to V_{\rho}$, $0 \leq r \leq d$
such that
\begin{equation}
f(\gamma\tau)=\sum_{r=0}^{d}j(\gamma,\tau)^{k-r}(j_{\gamma}')^{r}\rho(\gamma)f_{
r}(\tau) \label{qmod}
\end{equation}
for any $\gamma\in\Gamma$. The functions $f_{r}$ are independent of $\gamma$,
and we assume that $f_{d}\neq0$ (otherwise the depth of $f$ is smaller than
$d$). By taking $\gamma=\binom{1\ \ 0}{0\ \ 1}$ in Equation \eqref{qmod} we
obtain (under some simple normalization assumptions on $\rho$ and the powers of
$j(\gamma,\tau)$ if $k$ is not integral) that $f_{0}=f$. We denote the space of
quasi-modular forms (with the various differential properties and growth
conditions considered above) of weight $k$ and representation $\rho$ with
respect to $\Gamma$ by $\widetilde{\mathcal{M}}_{k}^{*}(\rho)$ (with the symbol
$*$ standing for one of the superscripts $an$, $sing$, $hol$, $mer$, $cusp$, or
$wh$), and $\widetilde{\mathcal{M}}_{k}^{*,\leq d}(\rho)$ denoting those
quasi-modular forms whose depth does not exceeds $d$.

If $\Gamma$ has cusps then certain assumptions on $\rho$ allow one to consider
Fourier expansions of modular and quasi-modular forms (as well as of the
functions $f_{r}$ in the latter case) around the cusps of $\Gamma$. We shall not
use these expansions in this paper.

\medskip

The tensor product of $f\in\mathcal{M}_{k}^{*}(\rho)$ and
$g\in\mathcal{M}_{l}^{*}(\eta)$, for $\rho$ and $\eta$ appropriate multiplier
systems of the same Fuchsian group $\Gamma$, lies in
$\mathcal{M}_{k+l}^{*}(\rho\otimes\eta)$. For $*=an$ or $*=sing$ the same
statement applies for modular forms with both holomorphic and anti-holomorphic
weights. Similarly, let $f$ and $g$ be elements in
$\widetilde{\mathcal{M}}_{k}^{*,\leq m}(\rho)$ and
$\widetilde{\mathcal{M}}_{l}^{*,\leq n}(\eta)$ with corresponding functions
$f_{r}$ for $0 \leq r \leq m$ and $g_{s}$ for $0 \leq s \leq n$ respectively.
Then, the tensor product $h=f \otimes g$ lies in
$\widetilde{\mathcal{M}}_{k+l}^{*,\leq m+n}(\rho\otimes\eta)$, with the
corresponding functions $h_{t}$, $0 \leq t \leq m+n$ being $\sum_{r+s=t}f_{r}
\otimes g_{s}$. Furthermore, the (tensor) product of \emph{almost holomorphic}
functions (namely, polynomials in $\frac{1}{2iy}$ over the ring of holomorphic
functions on $\mathcal{H}$) of depths (i.e., degrees as such polynomials) at
most $m$ and $n$ is almost holomorphic of depth not exceeding $m+n$. We shall
treat both the additive and multiplicative structures together in the main
result of this paper.

\subsection{Relations between Quasi-Modular Forms and Modular Forms}

Quasi-modular forms relate to modular forms according to the property presented
in the next Proposition. This result is stated in \cite{[KZ]} and \cite{[A]} and
proved in \cite{[MR]} for holomorphic quasi-modular forms of integral weights
with trivial representation with respect to $SL_{2}(\mathbb{Z})$ (see
Proposition 132 and Remarque 133 of \cite{[MR]}). However, it holds in a more
general context, an observation which will turn out useful later. We begin with
a result which is equivalent to Lemma 119 of \cite{[MR]}:
\begin{lem}
Take $f\in\widetilde{\mathcal{M}}_{k}^{*,\leq d}(\rho)$, with corresponding
functions $f_{r}$, $0 \leq r \leq d$. Then
$f_{r}\in\widetilde{\mathcal{M}}_{k-2r}^{*,\leq d-r}(\rho)$, with the
corresponding functions being $\binom{t}{r}f_{t}$ for any $r \leq t \leq d$.
\label{qmodcomp}
\end{lem}

The proof of Lemma 119 of \cite{[MR]} used the quasi-modularity of $E_{2}$ with
respect to $SL_{2}(\mathbb{Z})$. Since we work with a more general notion and
the Fuchsian group $\Gamma$ is also arbitrary, we extend his proof, with some
adjustments, to the present context.

\begin{proof}
Take $\tau\in\mathcal{H}$ and $\gamma$ and $\delta$ in $\Gamma$. Equation
\eqref{qmod} with $\gamma\delta\in\Gamma$ and $\tau\in\mathcal{H}$ yields
\[f(\gamma\delta\tau)=\sum_{r=0}^{d}\rho(\gamma\delta)f_{r}(\tau)j_{\gamma\delta
}^{k-r}(\tau)(j_{\gamma\delta}')^{r}.\] On the other hand, we can apply Equation
\eqref{qmod} with $\gamma\in\Gamma$ and $\delta\tau\in\mathcal{H}$, to obtain,
using Equation \eqref{cocycle} and its derivative, the equality
\[f(\gamma\delta\tau)=\sum_{r=0}^{d}\bigg[\sum_{s=r}^{d}\binom{s}{r}
\rho(\gamma)f_{s}(\delta\tau)j_{\delta}^{s+r-k}(\tau)(-j_{\delta}')^{s-r}\bigg]
j_{\gamma\delta}^{k-r}(\tau)(j_{\gamma\delta}')^{r}.\] Applying
$\rho(\gamma)^{-1}$, the fact that both equalities hold for every
$\gamma\in\Gamma$ implies
\begin{equation}
\rho(\delta)f_{r}(\tau)=\sum_{s=r}^{d}\binom{s}{r}f_{s}(\delta\tau)j(\delta,
\tau)^{s+r-k}(-j_{\delta}')^{s-r} \label{frqmod}
\end{equation}
for every $0 \leq r \leq d$.

We now use Equation \eqref{frqmod} in order to prove by decreasing induction on
$r$ the assertion of the lemma, namely
\[f_{r}(\delta\tau)=\sum_{t=r}^{d}\binom{t}{r}\rho(\delta)f_{t}(\tau)j(\delta,
\tau)^{k-t-r}(j_{\delta}')^{t-r}\] for every $0 \leq r \leq d$,
$\tau\in\mathcal{H}$, and $\delta\in\Gamma$. Substituting $r=d$ in Equation
\eqref{frqmod} gives the first step of the induction. Assuming that the
assertion holds for any $r<s \leq d$ we find that the left hand side of
Equation
\eqref{frqmod} with the index $r$ is
\[f_{r}(\delta\tau)j_{\delta}^{2r-k}(\tau)+\!\!\sum_{s=r+1}^{d}\!\!\binom{s}{r}
\!\sum_{t=s}^{d}\!\binom{t}{s}\rho(\delta)f_{t}(\tau)j_{\delta}^{k-t-s}(\tau)(j_
{\delta}')^{t-s}j_{\delta}^{s+r-k}(\tau)(-j_{\delta}')^{s-r}\!\!=\]
\[f_{r}(\delta\tau)j_{\delta}^{2r-k}(\tau)+\sum_{t=r+1}^{d}\binom{t}{r}
\rho(\delta)f_{t}(\tau)j_{\delta}^{r-t}(\tau)(j_{\delta}')^{t-r}\sum_{s=r+1}^{t}
\binom{t-r}{s-r}(-1)^{s-r}.\] The sum over $s$ equals $-1$ by the binomial
theorem. Comparing this expression with the left hand side of Equation
\eqref{frqmod} yields the assertion for $r$. Since the differential and growth
properties of the functions $f_{r}$ follow directly from those of $f$, this
proves the lemma.
\end{proof}

The relation between quasi-modular and modular forms is given in the following
generalization of Proposition 132 and Remarque 133 of \cite{[MR]} and of Theorem
1 of \cite{[A]}.

\begin{prop}
Let $f\in\widetilde{\mathcal{M}}_{k}^{an,\leq d}(\rho)$ with the corresponding
functions $f_{r}$, $0 \leq r \leq d$, and define the function
$F(\tau)=\sum_{r=0}^{d}\frac{f_{r}(\tau)}{(2iy)^{r}}$. Then
$F\in\mathcal{M}_{k}^{an}(\rho)$. Conversely, let
$F_{s}\in\mathcal{M}_{k-2s}^{an}(\rho)$ be given for every $0 \leq s \leq d$. In
this case the function $f(\tau)=\sum_{s=0}^{d}\frac{F_{s}(\tau)}{(-2iy)^{s}}$
lies in $\widetilde{\mathcal{M}}_{k}^{an,\leq d}(\rho)$, with the corresponding
functions being
$f_{r}(\tau)=\sum_{s=r}^{d}\binom{s}{r}\frac{F_{s}(\tau)}{(-2iy)^{s-r}}$ with $0
\leq r \leq d$ (hence the depth of $f$ is precisely $d$ if and only if
$F_{d}\neq0$). \label{qmodtomod}
\end{prop}

\begin{proof}
Recall that the equality $\Im\gamma\tau=\frac{y}{|j(\gamma,\tau)|^{2}}$ implies
\begin{equation}
\frac{1}{2i\Im\gamma\tau}=\frac{j(\gamma,\tau)^{2}}{2iy}-j(\gamma,\tau)j_{\gamma
}'. \label{y-1-1}
\end{equation}
Lemma \ref{qmodcomp} and Equation \eqref{y-1-1} allow one to write
$F(\gamma\tau)=\sum_{r=0}^{d}\frac{f_{r}(\gamma\tau)}{(2i\Im\gamma\tau)^{r}}$ as
\[\sum_{r=0}^{d}\sum_{t=r}^{d}\binom{t}{r}\rho(\gamma)f_{t}(\tau)j_{\gamma}^{
k-r-t}(\tau)(j_{\gamma}')^{t-r}\sum_{h=0}^{r}\binom{r}{h}\frac{j_{\gamma}^{2h}
(\tau)}{(2iy)^{h}}(-1)^{r-h}j_{\gamma}^{r-h}(\tau)(j_{\gamma}')^{r-h}\!=\]
\[=\sum_{0 \leq h \leq t \leq
d}\binom{t}{h}\frac{\rho(\gamma)f_{t}(\tau)}{(2iy)^{h}}j(\gamma,\tau)^{k-t+h}(j_
{\gamma}')^{t-h}\sum_{r=h}^{t}\binom{t-h}{r-h}(-1)^{r-h}.\] Since the sum over
$r$ is 1 if $t=h$ and 0 otherwise, this expression for $F(\gamma\tau)$ reduces
to $\sum_{t=0}^{d}\frac{\rho(\gamma)f_{t}(\tau)}{(2iy)^{t}}j(\gamma,\tau)^{k}
=j(\gamma,\tau)^{k}\rho(\gamma)F(\tau)$. This proves the first assertion. For
the second assertion, write
$f(\gamma\tau)=\sum_{s=0}^{d}\frac{F_{s}(\gamma\tau)}{(-2i\Im\gamma\tau)^{s}}$.
The modularity of the functions $F_{s}$ and Equation \eqref{y-1-1} now show that
this expression equals
\[\sum_{s=0}^{d}\rho(\gamma)F_{s}(\tau)j_{\gamma}^{k-2s}(\tau)\sum_{r=0}^{s}
\binom{s}{r}j_{\gamma}^{r}(\tau)(j_{\gamma}')^{r}\frac{j_{\gamma}^{2(s-r)}(\tau)
}{(-2iy)^{s-r}},\] which reduces to the desired expression by the definition of
the functions $f_{r}$. This proves the proposition.
\end{proof}

The two maps in Proposition \ref{qmodtomod} are inverse to one another in the
following sense. Given $f\in\widetilde{\mathcal{M}}_{k}^{an,\leq d}(\rho)$ hence
$f_{r}\in\widetilde{\mathcal{M}}_{k-2r}^{an,\leq d-r}(\rho)$ for each $0 \leq r
\leq d$ (by Lemma \ref{qmodcomp}), Proposition \ref{qmodtomod} constructs the
functions $F_{r}\in\mathcal{M}_{k-2r}^{an}(\rho)$ for every such $r$. Applying
the inverse construction from Proposition \ref{qmodtomod} to the functions
$F_{r}$ yields back the original quasi-modular form $f$. Conversely, assume that
$f\in\widetilde{\mathcal{M}}_{k}^{an,\leq d}(\rho)$ is obtained by Proposition
\ref{qmodtomod} from $F_{s}\in\mathcal{M}_{k-2s}^{an}(\rho)$, $0 \leq s \leq d$,
and let $f_{r}$, $0 \leq r \leq d$ be the corresponding functions from Equation
\eqref{qmod}. In this case the modular form in $\mathcal{M}_{k-2r}^{an}(\rho)$
constructed from $f_{r}\in\widetilde{\mathcal{M}}_{k-2r}^{an,\leq d-r}(\rho)$ is
the original modular form $F_{r}$ for each $0 \leq r \leq d$. Hence Proposition
\ref{qmodtomod} gives, for every weight $k$, depth bound $d$, group $\Gamma$,
and representation or multiplier system $\rho$, an isomorphism
\begin{equation}
\widetilde{\mathcal{M}}_{k}^{an,\leq
d}(\rho)\cong\bigoplus_{s=0}^{d}\mathcal{M}_{k-2s}^{an}(\rho), \label{qmaniso}
\end{equation}
which also preserves most reasonable growth conditions (boundedness, linear
exponential growth, exponential decay, etc.) at the cusps of $\Gamma$ (if
$\Gamma$ has cusps).

Proposition \ref{qmodtomod} and Equation \eqref{qmaniso} hold also when
replacing the superscript $an$ by $sing$ throughout. On the other hand,
restricting our attention to an element $f$ in the subspace
$\widetilde{\mathcal{M}}_{k}^{hol,\leq d}(\rho)$ of the isomorphism in Equation
\eqref{qmaniso}, the elements $F_{s}$ in the right hand side of that Equation
(and in particular $F=F_{0}$ from Proposition \ref{qmodtomod}) are almost
holomorphic functions on $\mathcal{H}$. On the other hand, if
$F\in\mathcal{M}_{k}^{an}(\rho)$ is almost holomorphic, we can write
$F(\tau)=\sum_{r=0}^{d}\frac{f_{r}(\tau)}{(2iy)^{r}}$ with $f_{r}$, $0 \leq r
\leq d$ holomorphic functions. Modularity implies
$F(\gamma\tau)=\rho(\gamma)j(\gamma,\tau)^{k}F(\gamma)$, and comparing the
polynomial expansion in $\frac{1}{2iy}$ over holomorphic functions on both sides
yields Equation \eqref{frqmod} for the functions $f_{r}$. Then the proof of
Lemma \ref{qmodcomp} and Proposition \ref{qmodtomod} imply that
$F_{r}(\tau)=\sum_{s=r}^{d}\binom{s}{r}\frac{f_{s}(\tau)}{(2iy)^{s-r}}$ is in
$\mathcal{M}_{k-2r}^{an}(\rho)$ (and almost holomorphic) for each $0 \leq r \leq
d$, $F_{0}=F$, and the function
$f_{0}:\tau\mapsto\sum_{r=0}^{d}\frac{F_{r}(\tau)}{(-2iy)^{r}}$ is in
$\widetilde{\mathcal{M}}_{k}^{hol,\leq d}(\rho)$. Hence this construction
reproduces (and generalizes) the isomorphism given in Proposition 132 and
Remarque 133 of \cite{[MR]} or in Theorem 1 of \cite{[A]}.

We remark that the action of the $\mathfrak{sl}_{2}$-triple presented in Section
3 of \cite{[A]} is a direct consequence of the holomorphic case of Proposition
\ref{qmodtomod} (with trivial representation) and the action of
$\mathfrak{sl}_{2}(\mathbb{C})$ on modular forms considered as functions on
$SL_{2}(\mathbb{R})$ with the appropriate behavior under the action of the
maximal compact subgroup $O(2)$ of $SL_{2}(\mathbb{R})$. To see this, observe
that \cite{[Ve]} shows that the elements $W=\binom{\ \ -i}{i\ \ \ }$,
$Z=\frac{1}{2}\binom{1\ \ \ \ i}{i\ \ -1}$, and
$\overline{Z}=\frac{1}{2}\binom{\ 1\ \ -i}{-i\ \ -1}$ of
$\mathfrak{sl}_{2}(\mathbb{C})$ form an $\mathfrak{sl}_{2}$-triple, and they act
on modular forms of weight $k$ as multiplication by $k$, the weight raising
operator $2i\delta_{k}=2i\frac{\partial}{\partial\tau}+\frac{k}{y}$, and the
weight lowering operator
$\frac{L}{2i}=-2iy^{2}\frac{\partial}{\partial\overline{\tau}}$, respectively.
Now, if $F$ arises from $f\in\widetilde{\mathcal{M}}_{k}^{hol}(\Gamma)$ as in
Proposition \ref{qmodtomod}, then $\delta_{k}F$ has ``constant coefficient''
$f'$ as a polynomial in $\frac{1}{2iy}$ (see, for example, Proposition 135 of
\cite{[MR]}). On the other hand,
$LF(\tau)=\sum_{r=0}^{d-1}\frac{(r+1)f_{r+1}(\tau)}{(2iy)^{r}}$ has ``constant
term'' $f_{1}$. Thus, the operators acting on
$\widetilde{\mathcal{M}}_{k}^{hol}(\Gamma)$ which are considered in \cite{[A]},
namely $H$ (multiplication by the weight), $D=\frac{\partial}{\partial\tau}$,
and $\delta:f \mapsto f_{1}$, correspond to the $\mathfrak{sl}_{2}$-triple $W$,
$\frac{Z}{2i}$, and $2i\overline{Z}$, respectively.

\section{Modular Forms with the Representation $V_{m}$ \label{Vm}}

In this Section we present modular forms with multiplier systems involving
$V_{m}$, establish the relation between these modular forms and quasi-modular
forms, and determine the space of such modular forms which are holomorphic.

\subsection{Quasi-Modular Forms as Modular Forms with Multiplier Systems
Involving $V_{m}$}

The group $SL_{2}(\mathbb{R})$ (and, more generally, $GL_{2}^{+}(\mathbb{R})$)
acts naturally on the space $\mathbb{C}^{2}$ of complex row vectors of length 2,
by $\gamma:u\mapsto\gamma u$. We denote this representation space $V_{1}$, and
let $V_{m}$ be its $m$th symmetric power. We denote the action of $\gamma \in
SL_{2}(\mathbb{R})$ on an element $u \in V_{m}$ by $\gamma^{m}u$ (this notation
should lead to no confusion with the $m$th power of $\gamma$ in $\Gamma$). We
shall also use multiplicative notation for elements
in $V_{m}$: If $v_{i}$, $1 \leq i \leq m$, are vectors in $\mathbb{C}^{2}$, we
denote the image of $v_{i}\otimes\ldots \otimes v_{m} \in V_{1}^{\otimes m}$ in
$V_{m}$ by the product $\prod_{i=1}^{m}v_{i}$, with powers denoting repetitions.

The modular forms with representation $V_{m}$, as appearing here, were defined
by Shimura in \cite{[Sh]} following an idea of Eichler in \cite{[E]}. In this
paper we show how these modular forms (and their direct limit) form a framework
for quasi-modular forms, and how almost holomorphic modular forms can be viewed
as holomorphic vector-valued modular forms. Since $V_{m}$ is a representation of
the full group $SL_{2}(\mathbb{R})$, we preserve $\Gamma$ in the notation
$\mathcal{M}_{k}^{*}(\Gamma,V_{m})$ etc. where there is no tensor product with
some multiplier system $\rho$ of $\Gamma$. We now present some of the properties
of modular forms with representations involving $V_{m}$, which will be needed
later on.

We start with the classical equation, stating that for every $\gamma \in
SL_{2}(\mathbb{R})$ and $\tau\in\mathcal{H}$, the equality
\begin{equation}
\gamma\left(\begin{array}{c} \tau \\
1\end{array}\right)=j(\gamma,\tau)\left(\begin{array}{c} \gamma\tau \\
1\end{array}\right) \label{V1mod}
\end{equation}
holds. Hence the map $\tau\mapsto\binom{\tau}{1}$ is in
$\mathcal{M}_{-1}^{hol}(\Gamma,V_{1})$ for any discrete subgroup $\Gamma$ of
$SL_{2}(\mathbb{R})$. Thus, $\tau\mapsto\binom{\overline{\tau}}{1}$ lies in
$\mathcal{M}_{0,-1}^{an}(\Gamma,V_{1})$ for any such $\Gamma$. It follows that
for any pair of numbers $m_{+}$ and $m_{-}$ such that $m_{+}+m_{-}=m$, the map
$\tau\mapsto\binom{\tau}{1}^{m_{+}}\binom{\overline{\tau}}{1}^{m_{-}}$ defines
an element of $\mathcal{M}_{-m_{+},-m_{-}}^{an}(\Gamma,V_{m})$, and for
$m_{+}=m$ and $m_{-}=0$ we can replace the superscript $an$ by $hol$. This
simple observation allows us to characterize modular forms with representations
involving $V_{m}$.

\begin{prop}
Let $\Gamma$ be a Fuchsian group, let $k$ and $l$ be weights, and let $\rho$ be
a multiplier system of weight $(k,l)$ for $\Gamma$. If
$f^{m_{+},m_{-}}\in\mathcal{M}_{k+m_{+},l+m_{-}}^{an}(\rho)$ for every pair
of non-negative integers $m_{+}$ and $m_{-}$ such that $m_{+}+m_{-}=m$ then the
$V_{m} \otimes V_{\rho}$-valued function
\begin{equation}
F(\tau)=\sum_{m_{+}+m_{-}=m}f^{m_{+},m_{-}}(\tau)\binom{\tau}{1}^{m_{+}}\binom{
\overline{\tau}}{1}^{m_{-}} \label{Vmdecom}
\end{equation}
lies in $\mathcal{M}_{k,l}^{an}(V_{m}\otimes\rho)$. Conversely, every element of
$\mathcal{M}_{k,l}^{an}(V_{m}\otimes\rho)$ is obtained in this way.
\label{Vmcompmod}
\end{prop}

\begin{proof}
The properties of $f^{m_{+},m_{-}}$ and the weights of the vectors
$\binom{\tau}{1}^{m_{+}}\binom{\overline{\tau}}{1}^{m_{-}}$ imply that each of
the terms in Equation \eqref{Vmdecom} lies in
$\mathcal{M}_{k,l}^{an}(V_{m}\otimes\rho)$. Hence
$F\in\mathcal{M}_{k,l}^{an}(V_{m}\otimes\rho)$ as well. On the other hand, for
any $V_{m} \otimes V_{\rho}$-valued function $F$, the fact that the elements
$\binom{\tau}{1}^{m_{+}}\binom{\overline{\tau}}{1}^{m_{-}}$ with $m_{+}+m_{-}=m$
form a basis for the space $V_{m}$ allows us to write $F(\tau)$ as in Equation
\eqref{Vmdecom}, with the function $f^{m_{+},m_{-}}$ taking values in $V_{\rho}$
for any pair $(m_{+},m_{-})$. Let now $\tau\in\mathcal{H}$ and $\gamma\in\Gamma$
be given, and write
\[F(\gamma\tau)=\sum_{m_{+}+m_{-}=m}f^{m_{+},m_{-}}(\gamma\tau)\binom{\gamma\tau
}{1}^{m_{+}}\binom{\overline{\gamma\tau}}{1}^{m_{-}}.\] On the other hand, under
the assumption
$F\in\mathcal{M}_{k,l}^{an}(V_{m}\otimes\rho)$ we must have
$F(\gamma\tau)=j(\gamma,\tau)^{k}\overline{j(\gamma,\tau)}^{~l}\big(\gamma^{m}
\otimes\rho(\gamma)\big)F(\tau)$,
and Equation \eqref{V1mod} implies that $F(\gamma\tau)$ equals
\[j(\gamma,\tau)^{k}\overline{j(\gamma,\tau)}^{\
l}\!\!\sum_{m_{+}+m_{-}=m}\!\!\rho(\gamma)f^{m_{+},m_{-}}(\tau)j(\gamma,\tau)^{
m_{+}}\overline{j(\gamma,\tau)}^{\
m_{-}}\binom{\gamma\tau}{1}^{m_{+}}\binom{\overline{\gamma\tau}}{1}^{m_{-}}.\]
As the two expressions for $F(\gamma\tau)$ are presented in the same basis for
$V_{m}$, comparing the coefficients in front of
$\binom{\gamma\tau}{1}^{m_{+}}\binom{\overline{\gamma\tau}}{1}^{m_{-}}$ implies
that $f^{m_{+},m_{-}}$ is in $\mathcal{M}_{k+m_{+},l+m_{-}}^{an}(\rho)$ for any
pair $(m_{+},m_{-})$. This proves the proposition.
\end{proof}

\medskip

For any $m$ and $n$ there corresponds a natural map $V_{m} \otimes V_{n} \to
V_{m+n}$, which takes $\prod_{i=1}^{m}u_{i}\otimes\prod_{j=1}^{n}v_{j}$ to
$\prod_{i=1}^{m}u_{i}\cdot\prod_{j=1}^{n}v_{j}$ (in the multiplicative
notation). This map is obtained as the projection from $V_{m} \otimes
V_{n}\cong\bigoplus_{j=0}^{\min\{m,n\}}V_{m+n-2j}$ onto the component $V_{m+n}$
of maximal dimension, or onto the $S_{m+n}$-invariant part, or modulo the action
of $S_{m+n}$. Therefore, we consider the (tensor) product of
$f\in\mathcal{M}_{k,l}^{an}(V_{m}\otimes\rho)$ with
$g\in\mathcal{M}_{r,s}^{an}(V_{n}\otimes\eta)$ as an element of
$\mathcal{M}_{k+r,l+s}^{an}(V_{m+n}\otimes\rho\otimes\eta)$ (with the same
convention for product of elements from $\mathcal{M}_{k}^{*}(V_{m}\otimes\rho)$
etc.). This way of treating products illustrates our way of viewing the
multiplicative structure of quasi-modular forms below in a clearer manner. This
multiplicative notation also allows us to obtain the following corollary of
Proposition \ref{Vmcompmod}, which will turn out useful when we consider
quasi-modular forms with arbitrary depth later.

\begin{cor}
If $F\in\mathcal{M}_{k-m,l}^{an}(V_{m}\otimes\rho)$ then
$i_{m}(F)=F\cdot\binom{\tau}{1}$ is an element of
$\mathcal{M}_{k-m-1,l}^{an}(V_{m+1}\otimes\rho)$. The map $i_{m}$ is injective.
As for the image of $i_{m}$, an element
$\widetilde{F}\in\mathcal{M}_{k-m-1,l}^{an}(V_{m+1}\otimes\rho)$ can be expanded
with respect to the basis
$\binom{\tau}{1}^{m_{+}}\binom{\overline{\tau}}{1}^{m_{-}}$ with
$m_{+}+m_{-}=m+1$ (as in Equation \eqref{Vmdecom}). Then $\widetilde{F}$ lies in
the image of $i_{m}$ if and only if the coefficient of
$\binom{\overline{\tau}}{1}^{m+1}$ in this expansion vanishes. \label{VmVm+1}
\end{cor}

\begin{proof}
Note that if $F$ is expanded as in Equation \eqref{Vmdecom} then
\[i_{m}(F)(\tau)=\sum_{m_{+}+m_{-}=m}f^{m_{+},m_{-}}(\tau)\binom{\tau}{1}^{m_{+}
+1}\binom{\overline{\tau}}{1}^{m_{-}}.\] The modularity of $i_{m}(F)$ now
follows from Proposition \ref{Vmcompmod} (or from the fact that $F$ is
multiplied by an element of $\mathcal{M}_{-1}^{hol}(\Gamma,V_{1})$), and the
injectivity of $i_{m}$ is clear. Now, images of $i_{m}$ have vanishing
coefficient in front of $\binom{\overline{\tau}}{1}^{m+1}$ in this expansion.
Conversely, let
$\widetilde{F}\in\mathcal{M}_{k-m-1,l}^{an}(V_{m+1}\otimes\rho)$, and expand
$\widetilde{F}$ as in Equation \eqref{Vmdecom} (with coefficients
$\widetilde{f}^{m_{+},m_{-}}$ with $m_{+}+m_{-}=m+1$). Assuming that
$\widetilde{f}^{0,m+1}=0$, we define
$f^{m_{+},m_{-}}=\widetilde{f}^{m_{+}+1,m_{-}}$ for any pair $(m_{+},m_{-})$
with $m_{+}+m_{-}=m$. Proposition \ref{Vmcompmod} implies that
$f^{m_{+},m_{-}}\in\mathcal{M}_{k-m_{-},l+m_{-}}^{an}(\rho)$. Hence the function
$F$ defined by Equation \eqref{Vmdecom} with the same functions
$f^{m_{+},m_{-}}$ lies in $\mathcal{M}_{k-m,l}^{an}(V_{m}\otimes\rho)$, and
$\widetilde{F}=i_{m}(F)$. This proves the corollary.
\end{proof}

Once again, Proposition \ref{Vmcompmod} and Corollary \ref{VmVm+1} extend to the
case where the superscript $an$ is replaced by $sing$. Moreover, the
holomorphicity of the function $\tau\mapsto\binom{\tau}{1}$ implies that for
$l=0$, Corollary \ref{VmVm+1} holds with
$\mathcal{M}_{k-m}^{an}(V_{m}\otimes\rho)$ replaced by any of the other spaces
of the form $\mathcal{M}_{k-m}^{*}(V_{m}\otimes\rho)$ defined above.

\medskip

Assume $l=0$ in Proposition \ref{Vmcompmod}. The action of powers of $y$ on
modular forms implies that a modular form
$F\in\mathcal{M}_{k-m}^{an}(V_{m}\otimes\rho)$ can be written as
\begin{equation}
F(\tau)=\sum_{s=0}^{m}\frac{F_{s}(\tau)}{(-2iy)^{s}}\binom{\tau}{1}^{m-s}
\binom{\overline{\tau}}{1}^{s}, \label{holwt}
\end{equation}
where the function $F_{s}$ lies in $\mathcal{M}_{k-2s}^{an}(\rho)$ for each $0
\leq s \leq m$. We can now state the first main result of this paper.

\begin{thm}
If $f\in\widetilde{\mathcal{M}}_{k}^{an,\leq m}(\rho)$ with corresponding
functions $f_{r}$, then the function
\begin{equation}
F(\tau)=\sum_{r}f_{r}(\tau)\binom{\tau}{1}^{m-r}\binom{1}{0}^{r} \label{exp10}
\end{equation}
lies in $\mathcal{M}_{k-m}^{an}(V_{m}\otimes\rho)$. Conversely, if
$F\in\mathcal{M}_{k-m}^{an}(V_{m}\otimes\rho)$, then we can expand it as in
Equation \eqref{exp10}, and the coefficient in front of $\binom{\tau}{1}^{m}$ is
an element of $f\in\widetilde{\mathcal{M}}_{k}^{an,\leq m}(\rho)$. Moreover,
every $F\in\mathcal{M}_{k-m}^{an}(V_{m}\otimes\rho)$ admits an expansion as in
Equation \eqref{holwt}, and $F$ corresponds to the quasi-modular form $f$ if and
only if the functions $F_{s}$ from Equation \eqref{holwt} and the functions
$f_{r}$ from Equation \eqref{qmod} are mutually related as in Proposition
\ref{qmodtomod}. \label{rels}
\end{thm}

\begin{proof}
By expanding the $r$th power of
$\binom{1}{0}=\frac{1}{-2iy}\big[\binom{\overline{\tau}}{1}-\binom{\tau}{1}\big]
$ in Equation \eqref{exp10} using the Binomial Theorem we obtain an expression
for $F(\tau)$ as in Equation \eqref{holwt}, with the function $F_{s}(\tau)$
being $\sum_{r=s}^{m}\binom{r}{s}\frac{f_{r}(\tau)}{(2iy)^{r-s}}$ for every $s$.
According to Lemma \ref{qmodcomp} and Proposition \ref{qmodtomod},
$F_{s}\in\mathcal{M}_{k-2s}^{an}(\rho)$ for every $s$, so that Proposition
\ref{Vmcompmod} and Equation \eqref{holwt} imply
$F\in\mathcal{M}_{k-m}^{an}(V_{m}\otimes\rho)$. Conversely, by writing $F$ as in
Equation \eqref{holwt} and expanding the $s$th power of
$\binom{\overline{\tau}}{1}=\binom{\tau}{1}-2iy\binom{1}{0}$ we obtain Equation
\eqref{exp10} with
$f_{r}(\tau)=\sum_{s=r}^{m}\binom{s}{r}\frac{F_{s}(\tau)}{(-2iy)^{s-r}}$, and
$f_{0}\in\widetilde{\mathcal{M}}_{k}^{an,\leq m}(\rho)$ by Proposition
\ref{qmodtomod}. It also follows from the proof of the preceding assertions that
the relations between the coefficients $F_{s}$ in Equation \eqref{holwt} and the
functions $f_{r}$ in Equation \eqref{exp10} in the expansions of the same
modular form $F$ agree with the relations stated in Proposition \ref{qmodtomod},
as is evident from. This proves the theorem.
\end{proof}

The holomorphicity of the basis specified in Equation \eqref{exp10} shows that
we can replace the superscript $an$ by any of the other superscripts defined
above, and the assertion of Theorem \ref{rels} extends to each of these cases.

Note that Theorem \ref{rels} also respects the multiplicative structures in the
following sense. For two elements $f\in\widetilde{\mathcal{M}}_{k}^{*,\leq
m}(\rho)$ and $g\in\widetilde{\mathcal{M}}_{l}^{*,\leq n}(\eta)$ and their
tensor product $h=f \otimes g\in\widetilde{\mathcal{M}}_{k+l}^{*,\leq
m+n}(\rho\otimes\eta)$, Theorem \ref{rels} yields the corresponding functions
$F\in\mathcal{M}_{k-m}^{*}(V_{m}\otimes\rho)$,
$G\in\mathcal{M}_{l-n}^{*}(V_{n}\otimes\eta)$, and $H$ in
$\mathcal{M}_{k+l-m-n}^{*}(V_{m+n}\otimes\rho\otimes\eta)$. Then $H=F \otimes G$
(under the convention of taking only the $V_{m+n}$ part of $V_{m} \otimes
V_{n}$),since if $f_{r}$, $0 \leq r \leq m$, $g_{s}$, $0 \leq s \leq n$, and
$h_{t}$, $0 \leq t \leq m+n$ are the functions corresponding to $f$, $g$, and
$h$ in Equation \eqref{qmod} respectively, then $h_{t}=\sum_{r+s=t}f_{r} \otimes
g_{s}$. Indeed, taking the tensor product of the functions $F$ and $G$ using the
expressions from Equation \eqref{exp10} yields the form of $H$ in the same
equation.

\medskip

Theorem \ref{rels} establishes a modular framework for quasi-modular forms with
bounded depth. The idea resembles slightly the
$\mathfrak{sl}_{2}(\mathbb{C})$-representations $\mathcal{U}_{k}$ presented at
the end of Section 3 of \cite{[A]}, though these are not equivalent
representations. We now use the maps $i_{m}$ from Corollary \ref{VmVm+1} in
order to gather all the quasi-modular forms together, and also to obtain the
multiplicative structure inside this ring. Let $V_{\infty}$ be the direct limit
of the $V_{m}$ with respect to the maps $i_{m}$. The representation space of
$V_{\infty}$ is infinite-dimensional, with two possible bases being
$\big\{\binom{\tau}{1}^{\infty-s}\binom{\overline{\tau}}{1}^{s}\big\}_{s=0}^{
\infty}$ and
$\big\{\binom{\tau}{1}^{\infty-r}\binom{1}{0}^{r}\big\}_{r=0}^{\infty}$.
However, images of this direct limit are $V_{\infty}$-valued functions of
$\tau\in\mathcal{H}$ in which only finitely many coefficients in either basis
are non-zero. After tensoring with $V_{\rho}$, the image of each of the spaces
$\mathcal{M}_{k-m}^{*}(V_{m}\otimes\rho)$ in this direct limit will be referred
to as modular forms (with the appropriate analytic properties) of weight
$k-\infty$ with representation $V_{\infty}\otimes\rho$ (denoted
$\mathcal{M}_{k-\infty}^{*}(V_{\infty}\otimes\rho)$), and any element
$\mathcal{M}_{k-\infty}^{*}(V_{\infty}\otimes\rho)$ arises from
$\mathcal{M}_{k-m}^{*}(V_{m}\otimes\rho)$ for some $m$. Even though
$k-\infty=-\infty$ for every finite $k$, we keep $k-\infty$ in the notation
$\mathcal{M}_{k-\infty}^{*}(V_{\infty}\otimes\rho)$, and the value of $k$ is of
course important. Since the maps $\widetilde{\mathcal{M}}_{k}^{*,\leq
m}(\rho)\to\mathcal{M}_{k-m}^{*}(V_{m}\otimes\rho)$ commute with the injections
$i_{m}$, we can consider the images of elements of
$\widetilde{\mathcal{M}}_{k}^{*}(\rho)$ (with no depth restriction) as modular
forms in $\mathcal{M}_{k-\infty}^{*}(V_{\infty}\otimes\rho)$.

The representation $V_{\infty}$ is also suitable for multiplicative properties.
Indeed, given $F\in\mathcal{M}_{k-m}^{*}(V_{m}\otimes\rho)$ and
$G\in\mathcal{M}_{l-n}^{*}(V_{n}\otimes\eta)$ with product $H=F \otimes G$ in
$\mathcal{M}_{k+l-m-n}^{*}(V_{m+n}\otimes\rho\otimes\eta)$, the equalities
\[i_{m}(F) \otimes G=i_{m+n}(H)=F \otimes i_{n}(G)\] hold in
$\mathcal{M}_{k+l-m-n-1}^{*}(V_{m+n+1}\otimes\rho\otimes\eta)$. Hence in the
direct limit we obtain a well-defined (tensor) product map from
$\mathcal{M}_{k-\infty}^{*}(V_{\infty}\otimes\rho)\otimes\mathcal{M}_{l-\infty}^
{*}(V_{\infty}\otimes\eta)$ to
$\mathcal{M}_{k+l-\infty}^{*}(V_{\infty}\otimes\rho\otimes\eta)$. As shown
above, this product map corresponds to the usual tensor of quasi-modular forms.
In particular, for integral $k$ and $l$ (or half-integral if $\Gamma$ is a
subgroup of $Mp_{2}(\mathbb{R})$) with trivial $\rho$ and $\eta$, we have a
well-defined product map, under which the vector space
$\bigoplus_{k}\mathcal{M}_{k-\infty}^{*}(\Gamma,V_{\infty})$ becomes a ring.
Theorem \ref{rels} and the following discussion show that this ring is
isomorphic to the ring $\bigoplus_{k}\widetilde{\mathcal{M}}_{k}^{*}(\Gamma)$ of
quasi-modular forms (with various analytic properties) with a trivial
representation with respect to $\Gamma$.

\subsection{Holomorphic Modular Forms with Multiplier Systems Involving $V_{m}$}

We now turn our attention to holomorphic and meromorphic forms with multiplier
systems involving $V_{m}$. According to Theorem \ref{rels} these objects are
analogous to holomorphic and meromorphic quasi-modular forms.

The spaces $\mathcal{M}_{k-m}^{*}(V_{m}\otimes\rho)$ are most intrinsically
described using the filtration arising from the maps $i_{m}$. For any $p \leq
m$, denote $\mathcal{M}_{k-m}^{*,p}(V_{m}\otimes\rho)$ the subspace of
$\mathcal{M}_{k-m}^{*}(V_{m}\otimes\rho)$ consisting of those elements whose
expansion as in Equation \eqref{holwt} involves non-zero coefficients $F_{s}$
only for $s \leq p$ (equivalently, in the expansion of Equation \eqref{exp10}
only non-zero terms $f_{r}$ with $r \leq p$ appear). Here we use an increasing
filtration, unlike the decreasing filtration of \cite{[Ve]} and others (though
this is essentially the same filtration). Multiple applications of Corollary
\ref{VmVm+1} show that $\mathcal{M}_{k-m}^{*,p}(V_{m}\otimes\rho)$ is the image
of $\mathcal{M}_{k-p}^{*}(V_{p}\otimes\rho)$ under $i_{m-1}\circ\ldots \circ
i_{p}$. The map $i_{m}$ takes $\mathcal{M}_{k-m}^{*,p}(V_{m}\otimes\rho)$
isomorphically onto $\mathcal{M}_{k-m-1}^{*,p}(V_{m+1}\otimes\rho)$ (for $p=m$
this is Corollary \ref{VmVm+1} again), which allows us to define
$\mathcal{M}_{k-\infty}^{*,p}(V_{\infty}\otimes\rho)$ in the direct limit. Now,
Theorem \ref{rels} implies that for an element of
$\mathcal{M}_{k-m}^{*,p}(V_{m}\otimes\rho)$ the coefficient $F_{p}$ in Equation
\eqref{holwt} and the coefficient $f_{p}$ in Equation \eqref{exp10} coincide,
and this common function lies in $\mathcal{M}_{k-2p}^{*}(\rho)$. Hence the space
$\mathcal{M}_{k-m}^{*,p}(V_{m}\otimes\rho)/\mathcal{M}_{k-m}^{*,p-1}(V_{m}
\otimes\rho)$ injects into $\mathcal{M}_{k-2p}^{*}(\rho)$. Moreover, $i_{m}$
defines an isomorphism between the latter quotient space and
$\mathcal{M}_{k-m-1}^{*,p}(V_{m+1}\otimes\rho)/\mathcal{M}_{k-m-1}^{*,p-1}(V_{
m+1}\otimes\rho)$, and the two injections commute with this isomorphism. Hence
we identify these isomorphic quotients, and use the notation
$\iota_{k,\rho}^{*,p}$ (which does not involve $m$) for this injection. We can
consider $\iota_{k,\rho}^{*,p}$ as defined on the direct limit
$\mathcal{M}_{k-\infty}^{*,p}(V_{\infty}\otimes\rho)/\mathcal{M}_{k-\infty}^{*,
p-1}(V_{\infty}\otimes\rho)$.

\medskip

We can now state the structure theorem for
$\mathcal{M}_{k-m}^{*}(V_{m}\otimes\rho)$ as well as for the direct limit
$\mathcal{M}_{k-\infty}^{*}(V_{\infty}\otimes\rho)$:

\begin{thm}
The map $\iota_{k,\rho}^{*,p}$ is an isomorphism in all cases, except for the
case where $\Gamma$ has no cusps, $*=hol$ (or equivalently $*=cusp=wh$), and the
weight $k$ is an integer between $p+1$ and $2p$ (hence $p>0$). Assuming that
$\Gamma$ has no cusps and $*=hol$, if $\rho$ is a representation factoring
though a finite quotient of $\Gamma$, then the only case where
$\iota_{k,\rho}^{*,p}$ is not an isomorphism is where $k=2p>0$ and
the representation $\rho$ contains a trivial factor. In this case
$\iota_{k,\rho}^{*,p}=0$ but maps into a non-zero space. \label{Vmmodfilt}
\end{thm}

The claim about cusp forms with trivial representation (for even $k$ and $m$)
appears in \cite{[Ve]}, and its proof (at least for weights 0 and 2) is given in
\cite{[KS]}. Here, we generalize it for all the various cases, and present
another proof for the case where $\Gamma$ has cusps. This new proof covers
certain situations to which the proof from \cite{[KS]} does not apply.

\smallskip

Before proving Theorem \ref{Vmmodfilt}, we introduce another basis for the
representation space of $V_{m}$ in some cases. Let $\varphi$ be an element of
$\widetilde{\mathcal{M}}_{2}^{mer,\leq1}(\Gamma)$ whose corresponding functions
in Equation \eqref{qmod} are $f_{0}=\varphi$ and $f_{1}=1$. Any normalized
logarithmic derivative of a (meromorphic) modular form of non-zero weight can be
taken as $\varphi$. In fact, Theorem 9 of \cite{[A]} shows that we can take
$\varphi$ with a unique simple pole on $Y(\Gamma)$ at any pre-fixed point in
$Y(\Gamma)$. Then the function $\varphi^{*}(\tau)=\varphi(\tau)+\frac{1}{2iy}$
lies in $\mathcal{M}_{2}^{sing}(\Gamma)$ (by Proposition \ref{qmodtomod}, for
example), and Theorem \ref{rels} shows that
\begin{equation}
w(\tau)=\varphi(\tau)\binom{\tau}{1}+\binom{1}{0}=\varphi^{*}(\tau)\binom{\tau
}{1}+\frac{1}{-2iy}\binom{\overline{\tau}}{1} \label{wdef}
\end{equation}
is in $\mathcal{M}_{1}^{mer}(\Gamma,V_{1})$. Using the basis $\binom{\tau}{1}$
and $w$ for $V_{1}$, an argument similar to Proposition \ref{Vmcompmod} proves
the following

\begin{prop}
An element $F\in\mathcal{M}_{k-m}^{sing}(V_{m}\otimes\rho)$ of weight $k$ and
representation $V_{m}\otimes\rho$ decomposes as
\begin{equation}
F(\tau)=\sum_{t=0}^{m}f_{t}^{w}(\tau)\binom{\tau}{1}^{m-t}w^{t}, \label{Vmw}
\end{equation}
where $f_{t}^{w}\in\mathcal{M}_{k-2t}^{sing}(\rho)$ for every $0 \leq t \leq m$.
Conversely, if $f_{t}^{w}\in\mathcal{M}_{k-2t}^{sing}(\rho)$ for every such $t$
then the function $F$ defined in Equation \eqref{Vmw} lies in
$\mathcal{M}_{k-m}^{sing}(V_{m}\otimes\rho)$. In addition,
$F\in\mathcal{M}_{k-m}^{mer}(V_{m}\otimes\rho)$ if and only if
$f_{t}^{w}\in\mathcal{M}_{k-2t}^{mer}(\rho)$ for all $t$. \label{Vmholmod}
\end{prop}

\begin{proof}
The equivalence of the modularity of $F$ and the modularity of the functions
$f_{t}^{w}$ is established as in the proof of Proposition \ref{Vmcompmod}. The
equivalence of the meromorphicity of $F$ and the meromorphicity of the functions
$f_{t}^{w}$ follows from the fact that the basis vectors
$\binom{\tau}{1}^{m-t}w^{t}$ are meromorphic functions of $\tau\in\mathcal{H}$.
This proves the proposition.
\end{proof}

Assume further that $\varphi\in\widetilde{\mathcal{M}}_{2}^{hol,\leq1}(\Gamma)$,
hence $w\in\mathcal{M}_{1}^{hol}(\Gamma,V_{1})$. Then the proof of Proposition
\ref{Vmholmod} implies the equivalence of
$F\in\mathcal{M}_{k-m}^{*}(V_{m}\otimes\rho)$ and
$f_{p}^{w}\in\mathcal{M}_{k-2p}^{*}(\rho)$ for all $p$ also for $*$ being one of
the superscripts $an$, $hol$, $cusp$, and $wh$. Hence every choice of such
$\varphi$ (and $w$) defines an isomorphism
\begin{equation}
\mathcal{M}_{k-m}^{*}(V_{m}\otimes\rho)\cong\bigoplus_{t=0}^{m}\mathcal{M}_{k-2t
}^{*}(\rho).
\label{Vmiso}
\end{equation}
The existence of $w\in\mathcal{M}_{1}^{hol}(\Gamma,V_{1})$ which is of the form
given in Equation \eqref{wdef} is equivalent to the existence of
$\varphi\in\widetilde{\mathcal{M}}_{2}^{hol,\leq1}(\Gamma)$ with the
corresponding function $f_{1}=1$. Such a quasi-modular form always exists if
$\Gamma$ has cusps. Theorem 4 of \cite{[A]} proves this assertion using
generalized Eisenstein series, though a more conceptual proof of this assertion
can be obtained using Serre duality: The filtration on
$\mathcal{M}_{1}^{hol}(\Gamma,V_{1})$ leads to a short exact sequence
\[0\to\mathcal{L}_{2}\to\mathcal{V}_{1,1}\to\mathcal{L}_{0}\to0\] of vector
bundles on $X(\Gamma)$ (where $\mathcal{L}_{k}$ is the line bundle corresponding
to modular forms of weight $k$ and $\mathcal{V}_{k,m}$ is the vector bundle of
rank $m+1$ corresponding to modular forms of weight $k$ and representation
$V_{m}$). $H^{1}\big(X(\Gamma),\mathcal{L}_{0}\big)$ is the space dual to
$H^{0}\big(X(\Gamma),\Omega^{1}=\mathcal{L}_{2}\otimes\mathcal{L}_{cusp}^{*}
\big)$, namely to $\mathcal{M}_{0}^{cusp}(\Gamma)$, and the latter space is 0 if
$\Gamma$ has cusps. Therefore, the corresponding sequence of global sections,
which is
\begin{equation}
0 \to \mathcal{M}_{2}^{hol}(\Gamma)\stackrel{i_{0}}{\to}
\mathcal{M}_{1}^{hol}(V_{1})\stackrel{\iota_{1,triv}^{hol,1}}{\to}\mathbb{C}\to0
\label{wto1}
\end{equation}
(where $triv$ is the trivial representation), is exact, and $w$ is any pre-image
of 1 under $\iota_{1,triv}^{hol,1}$. In the classical case of subgroups of
$SL_{2}(\mathbb{Z})$ we can take $\varphi$ to be $\frac{\pi i}{6}E_{2}$, a
multiple of the weight 2 quasi-modular Eisenstein series. Hence
$\varphi^{*}=\frac{\pi i}{6}\mathbf{E}_{2}$ is a multiple of the almost
holomorphic weight 2 Eisenstein series. On the other hand, if $\Gamma$ has no
cusps then this proof fails (since $H^{1}\big(X(\Gamma),\mathcal{L}_{0}\big)$ is
no longer trivial), and indeed
$\widetilde{\mathcal{M}}_{2}^{hol}(\Gamma)=\mathcal{M}_{2}^{hol}(\Gamma)$ in
this case (see Lemma 4 of \cite{[KS]} or Theorem 4 of \cite{[A]}).

\begin{proof}[Proof of Theorem \ref{Vmmodfilt}]
The assertion for $*=an$ and for $*=sing$ follows directly from Proposition
\ref{Vmcompmod}. Hence we restrict attention to meromorphic modular forms. Let
$f\in\mathcal{M}_{l}^{mer}(\rho)$. The operator
$\delta_{l}=\frac{\partial}{\partial\tau}+\frac{l}{2iy}$ takes $f$ to a modular
form of weight $l+2$, and the composition of $p$ consecutive such operators
takes $f$ to the element
\[\delta_{l}^{p}f(\tau)=\delta_{l+2p-2}\circ\ldots\circ\delta_{l}f(\tau)=\sum_{
r=0}^{p}\binom{p}{r}\prod_{j=1}^{r}(l+p-j)\frac{f^{(p-r)}(\tau)}{(2iy)^{r}}\] of
$\mathcal{M}_{l+2p}^{sing}(\rho)$. By Proposition 135 of \cite{[MR]}, the
operators $\delta_{l+2h}$, $0 \leq h \leq p-1$ commute with the ordinary
derivative $\frac{\partial}{\partial\tau}$ and the maps from Proposition
\ref{qmodtomod}. Thus, the $p$th derivative $f^{(p)}$ of $f$ is in
$\widetilde{\mathcal{M}}_{l+2p}^{mer,\leq p}(\rho)$ with the functions $f_{r}$
being $\binom{p}{r}\prod_{j=1}^{r}(l+p-j)f^{(p-r)}$ for each $0 \leq r \leq p$.
Theorem \ref{rels} now shows that
\[F(\tau)=\sum_{r=0}^{p}\binom{p}{r}\prod_{j=1}^{r}(l+p-j) \cdot
f^{(p-r)}(\tau)\binom{\tau}{1}^{m-r}\binom{1}{0}^{r}\] is in
$\mathcal{M}_{l+2p-m}^{mer}(V_{m}\otimes\rho)$ for any $m \geq p$. Take
$l=k-2p$, and assume first that none of the numbers $k-p-j$, $1 \leq j \leq p$,
vanish. Then the element $\frac{F}{\prod_{j=1}^{p}(k-p-j)}$ of
$\mathcal{M}_{k-m}^{mer,p}(V_{m}\otimes\rho)$ maps to $f$ under
$\iota_{k,\rho}^{mer,p}$. By replacing $mer$ by $hol$, $cusp$, or $wh$
(and $sing$ by $an$) the proof for any weight $k$ except for integral $k$
satisfying $p+1 \leq k \leq 2p$ is complete. In particular, the case $p=0$ is
completely covered by this proof, since there exists no $k$ satisfying $p+1 \leq
k \leq 2p$ in this case. This argument also covers the case $*=cusp$ if $\Gamma$
has cusps and $\rho$ factors through a finite quotient, since if $0 \neq
f\in\mathcal{M}_{k-2p}^{cusp}(\rho)$ in this case then $k>2p$. For the cases
$k=m$ and $k=m+2$ (with a trivial representation) this argument reproduces the
proof of \cite{[KS]}, since the matrix denoted $L_{m}(\tau)$ in that reference
maps $\binom{0}{1}^{m_{+}}\binom{1}{0}^{m_{-}}$ to
$\binom{\tau}{1}^{m_{+}}\binom{1}{0}^{m_{-}}$ as the $m$th
symmetric power of $\binom{1\ \ \tau}{0\ \ 1}$.

For the other cases we use Proposition \ref{Vmholmod}, which immediately proves
the case $*=mer$ for any group $\Gamma$. Moreover, given
$f\in\mathcal{M}_{k-2p}^{mer}(\rho)$ which is not holomorphic on $\mathcal{H}$,
Theorem 9 of \cite{[A]} allows us to choose an element
$\varphi\in\widetilde{\mathcal{M}}_{2}^{mer,\leq1}(\Gamma)$ having poles only at
the poles of $f$. In this case Equation \eqref{Vmw} yields a meromorphic
$\iota_{k,\rho}^{mer,p}$-pre-image of $f$ whose poles are only at the poles of
$f$. If $\Gamma$ has cusps then we choose holomorphic $\varphi$ and $w$ in
Equation \eqref{Vmw}. Then Proposition \ref{Vmholmod} and the additional
equivalences arising from this choice of $\varphi$ complete the proof for the
case of $\Gamma$ with cusps.

It remains to consider the image of $\iota_{k,\rho}^{*,p}$ for $*=hol=cusp=wh$
in the case where $\Gamma$ has no cusps and $k$ is an integer between $p+1$ and
$2p$. We first observe that $\rho$ is a representation (not a multiplier system)
since $k$ is an integer. Assuming that $\rho$ factors through a finite quotient
of $\Gamma$, we have $\mathcal{M}_{k-2p}^{hol}(\rho)=0$ for $k<2p$ (since the
weight is negative), and for $k=2p$ the space $\mathcal{M}_{k-2p=0}^{hol}(\rho)$
consists of constant functions. A constant in $V_{\rho}$ lies in
$\mathcal{M}_{0}^{hol}(\rho)$ if and only if it lies in the maximal subspace
$V_{\rho}^{triv}$ of $V_{\rho}$ on which $\rho$ acts trivially. Therefore all we
need to show is that if $p>0$, $k=2p$, and the space $V_{\rho}^{triv}$ is
non-trivial, then $\iota_{k,\rho}^{*,p}$ is still 0. Consider
$F\in\mathcal{M}_{2p-m}^{hol,p}(V_{m}\otimes\rho)$ (with $m \geq p$), and assume
that $\iota_{2p,\rho}^{*,p}(F)\neq0$ in
$\mathcal{M}_{0}^{hol}(\rho)=V_{\rho}^{triv}$. Let $\xi$ be a linear functional
on $V_{\rho}^{triv}$ which does not vanish on $\iota_{2p,\rho}^{*,p}(F)$ but
vanishes on the complement of $V_{\rho}^{triv}$ in $V_{\rho}$ (this is possible
since $\rho$ is essentially a representation of a finite group). Then
$\xi(F)\in\mathcal{M}_{2p-m}^{hol,p}(\Gamma,V_{m})$, and when we write $\xi(F)$
as in Equation \eqref{exp10}, the function $f_{p}$ is a non-zero constant
complex number. Theorem \ref{rels} and Lemma \ref{qmodcomp} then imply that
$f_{p-1}$ is an element of
$\widetilde{\mathcal{M}}_{2}^{hol}(\Gamma)\setminus\mathcal{M}_{2}^{hol}
(\Gamma)$ (here we use the assumption that $p>0$), in contradiction to Theorem 4
of \cite{[A]}. This shows that $\iota_{2p,\rho}^{*,p}=0$ in this case, which
completes the proof of the theorem.
\end{proof}

Note that the case $p=0$ in Theorem \ref{Vmmodfilt} reduces to the assertion
that $\iota_{k,\rho}^{*,0}$ defines an isomorphism between the space
$\mathcal{M}_{k-m}^{*,0}(V_{m}\otimes\rho)$ for any finite $m$, or equivalently
$\mathcal{M}_{k-\infty}^{*,0}(V_{\infty}\otimes\rho)$, and the space
$\mathcal{M}_{k}^{*}(\rho)$ for every $k$, $\Gamma$, $\rho$, and analytic type
$*$. Indeed, this is the isomorphism inverse to $i_{m-1}\circ\ldots \circ i_{0}$
(or to the direct limit map).

Consider now the case where $\Gamma$ has cusps. The proof of Theorem
\ref{Vmmodfilt} shows that a more general assertion is valid in this case: For
any cuspidal divisor $D$ on $X(\Gamma)$, the restriction of the map
$\iota_{k,\rho}^{wh,p}$ to modular forms whose poles are bounded by $D$ is an
isomorphism onto the subspace of $\mathcal{M}_{k-2p}^{wh}(\rho)$ consisting of
those modular forms whose polar divisor is bounded by $D$. Observe that the
degree of freedom in the choice of
$\varphi\in\widetilde{\mathcal{M}}_{2}^{hol,\leq1}(\Gamma)$ (hence $w$) in this
case corresponds to the addition of an element
$h\in\mathcal{M}_{2}^{hol}(\Gamma)$ to $\varphi$ (hence adding
$h\binom{\tau}{1}$ to $w$)---see Equation  \eqref{wto1}, for example. In
particular, if $\mathcal{M}_{2}^{hol}(\Gamma)=0$  (this is the situation, for
example, when $\Gamma=SL_{2}(\mathbb{Z})$), then $\varphi$ and $w$ are unique,
and the decomposition of $\mathcal{M}_{k-m}^{*}(V_{m}\otimes\rho)$ into
$\bigoplus_{p=0}^{m}\mathcal{M}_{k-2p}^{*}(\rho)$, given by Equations
\eqref{Vmw} and \eqref{Vmiso}, is canonical. Returning to the general case with
$\varphi\in\widetilde{\mathcal{M}}_{2}^{mer,\leq1}(\Gamma)$, we remark that
$\mathcal{M}_{k-m}^{sing,p}(V_{m}\otimes\rho)$ and
$\mathcal{M}_{k-m}^{mer,p}(V_{m}\otimes\rho)$ consist of those elements $F$ for
which the decomposition in Equation \eqref{Vmw} involves non-zero functions
$f_{t}^{w}$ only for $t \leq p$. The same assertion holds for the other spaces
$\mathcal{M}_{k-m}^{*}(V_{m}\otimes\rho)$ if
$\varphi\in\widetilde{\mathcal{M}}_{2}^{hol,\leq1}(\Gamma)$, and extends to
$\mathcal{M}_{k-\infty}^{*}(V_{\infty}\otimes\rho)$.

\medskip

We recall that if the dimension of $V_{\rho}$ is finite then the spaces
$\mathcal{M}_{k}^{hol}(\rho)$ are finite-dimensional. Indeed, if $k$ is integral
(or half-integral for $\Gamma \subseteq Mp_{2}(\mathbb{R})$) and $\rho$ is a
representation, then $\mathcal{M}_{k}^{hol}(\rho)$ is the space of global
sections of a vector bundles (of finite rank) over the compact Riemann surface
$X(\Gamma)$. For the case of multiplier systems see, for example, Proposition 9
of \cite{[KM]}. If $\Gamma$ has cusps then the same assertion holds for the
spaces $\mathcal{M}_{k}^{cusp}(\rho)$, and more generally to every subspace of
$\mathcal{M}_{k}^{wh}(\rho)$ defined by a bound on the polar divisor. Hence
Theorem \ref{Vmmodfilt} has the following

\begin{cor}
If $k$ is not an integer between 1 and $2m$ then
\begin{equation}
\dim\mathcal{M}_{k-m}^{hol}(V_{m}\otimes\rho)=\sum_{t=0}^{m}\dim\mathcal{M}_{
k-2t}^{hol}(\rho).
\label{dimeq}
\end{equation}
If $\Gamma$ has cusps then Equation \eqref{dimeq} holds in general, as well as
the corresponding assertion for the spaces $\mathcal{M}^{cusp}$ and for
subspaces of $\mathcal{M}^{wh}$ in which the polar divisor is bounded by a fixed
cuspidal divisor. If $\Gamma$ has no cusps, $k$ is an integer between 1 and
$2m$, and the representation $\rho$ factors through a finite quotient of
$\Gamma$, then Equation \eqref{dimeq} still holds if $k$ is odd. If $k=2p$ for
some $0<p \leq m$ then
\[\dim\mathcal{M}_{2p-m}^{hol}(V_{m}\otimes\rho)=\sum_{t=0}^{p-1}\dim\mathcal{M}
_{2p-2t}^{hol}(\rho),\]
and the right hand side of Equation \eqref{dimeq} is obtained from this common
value by adding $\dim\mathcal{M}_{0}^{hol}(\rho)=\dim V_{\rho}^{triv}$.
\end{cor}

\noindent\textsc{Fachbereich Mathematik, AG 5, Technische Universit\"{a}t
Darmstadt, Schlossgartenstrasse 7, D-64289, Darmstadt, Germany}

\noindent E-mail address: zemel@mathematik.tu-darmstadt.de

\end{document}